\documentclass[review,onefignum,onetabnum]{siamart190516}

\usepackage{lipsum}
\usepackage{amsfonts}
\usepackage{graphicx}
\usepackage{epstopdf}
\usepackage{algorithmic}
\ifpdf
  \DeclareGraphicsExtensions{.eps,.pdf,.png,.jpg}
\else
  \DeclareGraphicsExtensions{.eps}
\fi
\usepackage[utf8]{inputenc}
\usepackage{theorem,ifthen,algorithm,algorithmic}
\usepackage{amssymb,amsfonts,amsmath,latexsym,dsfont}
\usepackage{tikz}
\usetikzlibrary{plotmarks}
\usepackage{standalone}
\usepackage{graphicx}
\usepackage{mathrsfs}
\usepackage{cleveref}
\usepackage{mathtools}
\usepackage{multirow}
\usepackage{bm}
\usepackage{bbm}
\usepackage{subcaption} 
\usepackage[normalem]{ulem}

\DeclareMathOperator*{\argmin}{arg\,min}
\graphicspath{{figures/}{./}}

\newdimen\iwidth
\newdimen\iheight

\usepackage{siunitx}
\newsiamremark{remark}{Remark}
\newsiamremark{hypothesis}{Hypothesis}
\crefname{hypothesis}{Hypothesis}{Hypotheses}
\newsiamthm{claim}{Claim}

\headers{PNKH-B: A Projected Newton-Krylov Method}{Kelvin Kan, Samy Wu Fung and Lars Ruthotto}

\title{PNKH-B: A Projected Newton-Krylov Method for Large-Scale Bound-Constrained Optimization\thanks{Submitted to the editors DATE.
\funding{KK and LR are supported by the U.S. National Science Foundation (NSF) Grant No.  DMS 1751636.
\newline
SWF is supported by AFOSR MURI FA9550-18-1-0502, AFOSR Grant No. FA9550-18-1-0167, and ONR Grant No. N00014-18-1-2527.}}}

\author{Kelvin Kan\thanks{Department of Mathematics, Emory University, USA 
  (\email{kelvin.kan@emory.edu}).}
\and Samy Wu Fung\thanks{Department of Mathematics, University of California, Los Angeles, CA, USA 
  (\email{swufung@math.ucla.edu}).}
\and Lars Ruthotto\thanks{Departments of Mathematics and Computer Science, Emory University, USA (\email{lruthotto@emory.edu}).}}

\usepackage{amsopn}

\newcommand{\hf}{{\frac 12}}

\newcommand{\bfA}{{\bf A}}

\newcommand{\bfC}{{\bf C}}
\newcommand{\bfD}{{\bf D}}
\newcommand{\bfE}{{\bf E}}
\newcommand{\bfF}{{\bf F}}
\newcommand{\bfG}{{\bf G}}
\newcommand{\bfH}{{\bf H}}
\newcommand{\bfI}{{\bf I}}
\newcommand{\bfJ}{{\bf J}}
\newcommand{\bfK}{{\bf K}}
\newcommand{\bfL}{{\bf L}}
\newcommand{\bfM}{{\bf M}}

\newcommand{\bfP}{{\bf P}}
\newcommand{\bfQ}{{\bf Q}}
\newcommand{\bfR}{{\bf R}}
\newcommand{\bfS}{{\bf S}}
\newcommand{\bfT}{{\bf T}}
\newcommand{\bfU}{{\bf U}}
\newcommand{\bfV}{{\bf V}}
\newcommand{\bfW}{{\bf W}}
\newcommand{\bfX}{{\bf X}}

\newcommand{\bfb}{{\bf b}}
\newcommand{\bfc}{{\bf c}}
\newcommand{\bfe}{{\bf e}}

\newcommand{\bfx}{{\bf x}}
\newcommand{\bfy}{{\bf y}}
\newcommand{\bfu}{{\bf u}}

\newcommand{\bfp}{{\bf p}}

\newcommand{\bfd}{{\bf d}}
\newcommand{\bfm}{{\bf m}}
\newcommand{\bfr}{{\bf r}}

\newcommand{\bfv}{{\bf v}}
\newcommand{\bfw}{{\bf w}}

\newcommand{\bfl}{{\bf l}}
\newcommand{\bfz}{{\bf z}}

\newcommand{\calA}{{\cal A}}
\newcommand{\CO}{{\cal O}}

\newcommand{\R}{\ensuremath{\mathds{R}}}

\newtheorem{example}{Example}

\ifpdf
\hypersetup{
  pdftitle={PNKH-B: A Projected Newton-Krylov Method for Large-Scale Bound-Constrained Optimization},
  pdfauthor={Kelvin Kan, Samy Wu Fung, and Lars Ruthotto}
}
\fi

\usepackage{pgfplots}

\begin{document}

\maketitle
\begin{abstract}
We present PNKH-B, a projected Newton-Krylov method for iteratively solving large-scale optimization problems with bound constraints. PNKH-B is geared toward situations in which function and gradient evaluations are expensive, and the (approximate) Hessian is only available through matrix-vector products. This is commonly the case in large-scale parameter estimation, machine learning, and image processing. In each iteration, PNKH-B uses a low-rank approximation of the (approximate) Hessian to determine the search direction and construct the metric used in a projected line search. The key feature of the metric is its consistency with the low-rank approximation of the Hessian on the Krylov subspace. This renders PNKH-B similar to a projected variable metric method. We present an interior point method to solve the quadratic projection problem efficiently. Since the interior point method effectively exploits the low-rank structure, its computational cost only scales linearly with respect to the number of variables, and it only adds negligible computational time. We also experiment with variants of PNKH-B that incorporate estimates of the active set into the Hessian approximation. We prove the global convergence to a stationary point under standard assumptions. Using three numerical experiments motivated by parameter estimation, machine learning, and image reconstruction, we show that the consistent use of the Hessian metric in  PNKH-B leads to fast convergence, particularly in the first few iterations. We provide our MATLAB implementation at \url{https://github.com/EmoryMLIP/PNKH-B}.
\end{abstract}

\begin{keywords}
projected Newton-Krylov method, bound-constrained optimization, large-scale optimization
\end{keywords}

\begin{AMS}
49M15, 90C06, 90C25, 90C26, 90C51
\end{AMS}

\section{Introduction}
We introduce PNKH-B, a projected Newton-Krylov method for approximately solving large-scale optimization problems with bound constraints such as
\begin{equation}\label{eq:optProb}
\min_{\bfx} f(\bfx) \quad \text{subject to} \quad \bfl \leq \bfx \leq  \bfu.
\end{equation}
Here,  $f: \R^n \to \R$ is twice differentiable, the inequalities are applied component-wise, and the vectors $\bfl, \bfu\in\R^n \cup \{ \pm \infty \}$ with $\bfl \leq \bfu$ define the box $C=[\bfl,\bfu]$.
While our method also applies to small and medium scale problems, we focus in this paper on situations in which evaluating $f$ and its gradient is computationally expensive and the (approximate) Hessian is only available through matrix-vector products. This is common in PDE parameter estimation~\cite{dey1979,Flath:2011gm,Haber2015,mcgillivray1992,ruthotto2017jinv,wufung2019thesis,wufung2019uncertainty}, image processing~\cite{chan2020two,Hu2019,hu2019spectral,hu2019nonlinear,wang2019AO,wang2019,fung2020multigrid}, neural networks~\cite{chang2018,fung2019admm,haber2019imexnet,haber2017stable,haber2018learning,ruthotto2019deep}, etc. PNKH-B aims to approximately solve the problem using only a few function and gradient evaluations and Hessian-vector products.

Our PNKH-B scheme is motivated by projected variable metric methods and its  $k$th  iteration reads
\begin{equation}\label{eq:iterHess}
\bfx_{k+1} = \Pi_{\|\cdot\|_{\tilde{\bfH}_{k}}}\left( \bfy_{k+1}  \right), \quad \text{ with } \quad \bfy_{k+1} = \bfx_k - \mu_k \bfH_k^{-1} \nabla f(\bfx_k),
\end{equation}
where $\bfH_k$ is a low-rank approximation of the Hessian, $\tilde{\bfH}_k$ is a symmetric positive definite matrix that equals $\bfH_k$ on its Krylov subspace, $\mu_k$ is a suitable step size, $\Pi_{\|\cdot\|_{\tilde{\bfH}_{k}}}$ is the projection operator onto the box $C$ with the metric induced by $\tilde{\bfH}_k$, that is,
\begin{equation}\label{eq:projHess}
\Pi_{\|\cdot\|_{\tilde{\bfH}_{k}}}(\bfy) = \argmin_{\bfz} \hf\|\bfz-\bfy\|_{\tilde{\bfH}_k}^2 \quad \text {subject to} \quad \bfl \leq \bfz \leq  \bfu.
\end{equation}
We slightly abuse notation and denote the pseudoinverse of $\bfH_k$ by $\bfH_k^{-1}$ in order to be consistent with the conventional notation used in Newton's method. In our experiments, we obtain $\tilde{\bfH}_k$ by adding a small shift to the nullspace of $\bfH_k$.
The norm induced by the approximate Hessian is defined by 
$\|\bfv\|_{\tilde{\bfH}_k} :=\sqrt{ \bfv^\top \tilde{\bfH}_k \bfv}$ for all $\bfv \in \mathbb{R}^n$.
PNKH-B is called a generalized variable metric (or generalized one-metric) method because the variable metric used to determine the search direction and the projection are equivalent on the subspace spanned by $\bfH_k$.

The projection problem~\eqref{eq:projHess} is a convex quadratic program, which has no closed-form solution in general.
Exceptions are projected gradient methods that are obtained by using a diagonal Hessian approximation, where the projection simplifies to
\begin{equation}\label{eq:ProjEuclidean}
\Pi_{\|\cdot\|_{2}}(\bfy) = \max\{\min\{\bfy,\bfu\},\bfl\}.
\end{equation}
While being simple and robust, the projected gradient method is not a method of choice in our setting due to the large number of iterations needed to obtain a reasonable solution~\cite{Dunn1980}.
Similarly, the convergence of other first-order methods such as AdaGrad~\cite{Duchi2011}, which uses a diagonal Hessian approximation, is also not fast enough for large-scale problems.
On the other extreme using a projected Newton step (obtained using $\bfH_k= \nabla^2 f(\bfx_k)$) is intractable due to the computational cost of solving the large-scale quadratic program involved in the projection~\eqref{eq:projHess}.

In this paper, we propose PNKH-B, an iterative method for large-scale bound-constrained optimization. PNKH-B builds a low-rank approximation of the (approximate) Hessian by using Lanczos tridiagonalization to compute a basis of the Krylov subspace defined by the (approximate) Hessian and gradient at the $k$th step.
PNKH-B then uses the low-rank approximation to
compute the search direction and define the norm of the projection. To ensure that the norm of the projection is well-defined, a small shift is added in the orthogonal complement of the subspace spanned by the low-rank approximation of the Hessian. By construction, both matrices induce the same metric on the Krylov subspace.
Therefore, we consider PNKH-B to be a generalized one-metric method. Under mild assumptions, we prove global convergence of our method.
A main contribution is an interior-point method for computing the projection step~\eqref{eq:projHess} that, by exploiting the structure of $\tilde{\bfH}_k$, only adds negligible computational costs compared to~\eqref{eq:ProjEuclidean}. 
Although not shown here, our method can be straightforwardly extended to other low-rank representations, e.g., arising in L-BFGS~\cite{Byrd:1995iv,dener2019accelerating,Gilbert2006}. However, we focus on inexact Newton methods  because we aim to solve large-scale problems using as few evaluations as possible and quasi-Newton methods generally require more iterations to convergence than inexact Newton methods~\cite[Chapter 6]{Haber2015}.
We demonstrate the effectiveness of PNKH-B on three numerical experiments that are motivated by PDE parameter estimation, machine learning, and image reconstruction.

The remainder of this paper is organized as follows. In \Cref{sec:related_work}, we give an overview of the related work. In \Cref{sec:projectedNewtonMethod}, we describe our PNKH-B. In \Cref{sec:proof_convergence}, we provide theoretical guarantees. In \Cref{sec:exp_results}, we present three experimental results to illustrate the effectiveness of our method. We conclude with a discussion and future outlooks in \Cref{sec:conclusion}.

\section{Related Work}\label{sec:related_work}
Projected inexact Newton and quasi-Newton schemes are among the most effective and commonly used solvers for constrained large-scale optimization problems such as~\eqref{eq:optProb}. They have been studied and applied extensively in the past four decades, e.g., \cite{Bertsekas1982,Byrd:1995iv,Schmidt2009,schmidt2012}. In this section, we give a brief overview of the schemes that bear similarities with our approach. 

\begin{figure}
\centering
\includegraphics[width=1\textwidth]{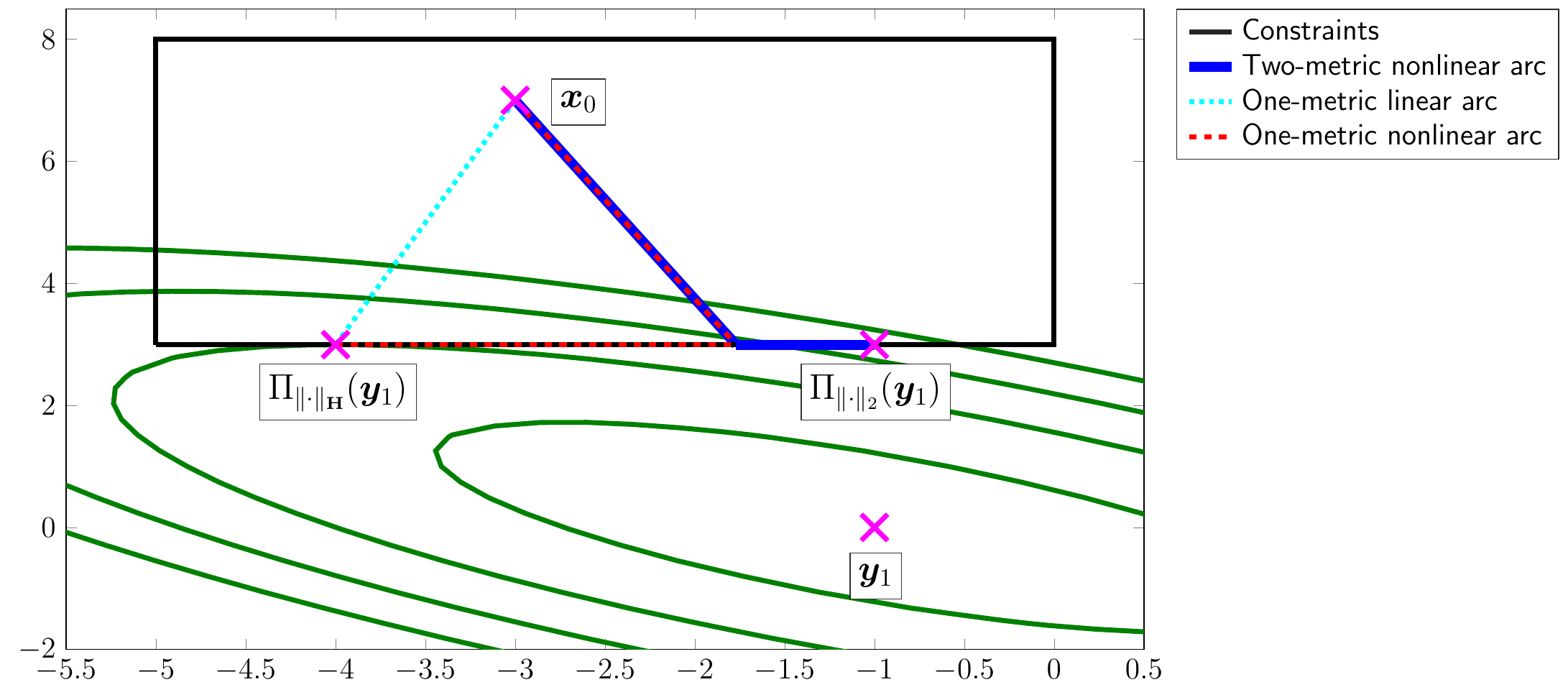}
\caption{An illustration of the first iterations of different methods on the quadratic optimization problem (\ref{eq:toy_problem}), where $\bfH=[1,1;1,2]$, $\bfb=[1;1]$, $\bfl=[-5;3]$, $\bfu=[0;8]$, $\bfx_0=[-3;7]$, $\bfy_1 = \bfx_0 - \bfH^{-1}\nabla f(\bfx_0)=- \bfH^{-1}\bfb=[-1;0]$ is the updated variable before projection, $\Pi_{\| \cdot \|_{\bfH}} ({\bfy_1})=[-4;3]$ is the projection with the Hessian metric and is the optimal solution, and $\Pi_{\|\cdot\|_{2}} ({\bfy_1})=[-1;3]$ is the projection with the Euclidean metric. The linear/piecewise linear line search arcs for one/two-metric methods are shown. The one-metric piecewise linear arc, which is used in our proposed PNKH-B, is the best one as it searches along the boundary and gives the optimal solution. The one-metric linear arc is less natural, does not search along the boundary and gives suboptimal iterate whenever step size 1 is not used. Finally, the two-metric piecewise linear arc searches for the opposite direction of the one-metric piecewise linear arc. It gives a suboptimal iterate $\Pi_{\|\cdot\|_{2}} ({\bfy_1})$ and it will be stuck at $\Pi_{\|\cdot\|_{2}} ({\bfy_1})$ even when the exact Hessian is used, i.e., it generates $\Pi_{\|\cdot\|_{2}} ({\bfy_k})=\Pi_{\|\cdot\|_{2}} ({\bfy_1})$ for all $k\geq 2$. \label{fig:projections}}
\end{figure}	

One property that can be used to group existing schemes is the metric used to determine the search direction and the projection.
We refer to schemes that use the same metric for both steps (e.g.,~\eqref{eq:iterHess}) as one-metric schemes and schemes that use different metrics for the two steps as two-metric schemes. Another distinguishing feature is the order in which projections and line searches are performed.
Schemes that project each line search iterate (e.g.,~\eqref{eq:iterHess}) in general lead to a piecewise linear arc, while applying the projection only once before the line search yields a linear arc. 
In the following, we will review existing methods according to these choices and provide an example to highlight their differences.

An extensively studied two-metric scheme \cite{Bertsekas1982,Gafni1984} uses a search direction induced by an approximated Hessian norm and a projection with respect to the Euclidean metric. That is, its formulation in each iteration is given by~\eqref{eq:iterHess} and ~\eqref{eq:projHess} except that the projection uses the Euclidean metric as this provides a closed-form to the projection problem using~\eqref{eq:ProjEuclidean}. Its line search induces a piecewise linear arc, see \Cref{fig:projections}. However, its convergence is not guaranteed in general, see \Cref{example}, \cite{Bertsekas1982} and \cite[Chapter~5.5.1]{Kelley:1999vd}. The global convergence of this approach for convex problems with linear constraints was proven by \cite{Bertsekas1982,Gafni1984} when a variable partitioning scheme is used. It partitions the components of the $k$th iterate into an active set in which the components are at or close to the boundary of the feasible set and an inactive set in which the components are in the interior of the feasible set. A search direction induced by the Euclidean norm is used for the active components and a search direction induced by $\| \cdot\|_{\bfH_k}$ is used for the inactive components. They also prove local superlinear convergence under certain conditions. Since then, variable partitioning has become a recurring theme for two-metric schemes \cite{Haber2015,herring2019,KimEtAl2010,landi2008projected,landi2012improved,landi2013nptool,schmidt2012}. Although the projection of the convergent two-metric scheme~\eqref{eq:ProjEuclidean} can be computed immediately, it requires appropriate scaling for the Euclidean norm induced search direction before combining the two search directions. 
Moreover, when many constraints are active,  two-metric schemes essentially become projected gradient methods. A specific drawback in Newton-Krylov schemes for large-scale problems is that the partitioning of the variables complicates the design of effective preconditioners. Given a preconditioner $\bfM_k$ for the approximate Hessian $\bfH_k$ and $\bfP_k$ the projection operator onto the inactive set at the $k$th step, the most natural choice is to precondition $\bfP_k^\top \bfH_k \bfP_k$ by $\bfP_k^\top \bfM^{-1}_k \bfP_k$. Since it is intractable to compute $(\bfP_k^\top \bfM^{-1}_k \bfP_k)^{-1}$, one might use the approximation $\bfP_k^\top \bfM_k \bfP_k$. However, note that $(\bfP_k^\top \bfM^{-1}_k \bfP_k)^{-1} \neq \bfP_k^\top \bfM_k \bfP_k$ in general.

Another well-studied approach is the one-metric scheme with linear arc. It is generally performed as follows. 
At the $k$th iteration, it (approximately) solves~\eqref{eq:iterHess} with $\mu_k=1$ to obtain a projection. Then it performs a line search along the straight line connecting the current iterate $\bfx_k$ and the projection; this linear line search is done in order to limit the number of solving costly projections.
This scheme is studied with different approximations of the Hessian, solvers for the projection~\eqref{eq:iterHess} with $\mu_k=1$ or backtracking schemes to determine the next iterate $\bfx_{k+1}$. For instance, the widely-applied L-BFGS-B \cite{Byrd:1995iv,Morales2011} uses the limited-memory BFGS matrix for $\bfH_k$ and approximately solves the projection~\eqref{eq:iterHess} with $\mu_k=1$ without any constraints, then it truncates the path toward the solution in order to satisfy the constraints. Finally it backtracks along the straight line to obtain $\bfx_{k+1}$. Other variants of this one-metric method with linear arc include \cite{Becker2012,becker2018,Hager2006,lee2014,Schmidt2009}. Although the consistent choice of metric could generate a better update direction than the two-metric scheme, it results in a suboptimal iterate which does not lie in the boundary whenever a step size of 1 is not used. Also, because the line search is just simply along the straight line connecting the previous iterate $\bfx_k$ and the projection, it can result in an inferior iterate $\bfx_{k+1}$ when compared to a more natural line search along a piecewise linear arc used in our method, see \Cref{fig:projections}. 
\begin{example}\label{example}
We illustrate the differences between one-metric and two-metric schemes with linear and piecewise linear arcs, respectively, using a two-dimensional quadratic program
\begin{equation}\label{eq:toy_problem}
    \min_{\bfx} \frac{1}{2}\bfx^\top \bfH \bfx + \bfb^\top \bfx \quad \text{subject to} \quad \bfl \leq \bfx \leq \bfu.
\end{equation}
The first iteration before projection of the one-metric method with piecewise linear arc reads
\begin{equation*}
    \bfy_1(\mu) = \bfx_0 - \mu \bfH^{-1}\nabla f(\bfx_0) = (1-\mu)  \bfx_0 - \mu \bfH^{-1}\bfb,
\end{equation*}
where $\mu$ is a step size determined by a backtracking line search scheme. The projection with the Hessian metric is given by
\begin{equation}\label{eq:motivation_exmaple}
\begin{split}
    \Pi_{\| \cdot \|_{\bfH}} (\bfy_1(\mu)) 
    & = \argmin_{\bfl \leq \bfz\leq \bfu} \frac{1}{2} \bfz^\top \bfH \bfz - (1-\mu) \bfz^\top \bfH \bfx_0 + \mu \bfb^\top \bfz.
    \end{split}
\end{equation}
When $\mu=1$, i.e., the first step of the backtracking line search, the projection problem is equivalent to the original optimization problem. So the backtracking line search stops at the first step and the one-metric method with the nonlinear line search converges in one iteration. This is because the Hessian metric projection is consistent with the steepest descent direction $\bfH^{-1} \nabla f$ induced by the Hessian metric. 
If for a non-quadratic objective function, the initial step size is not accepted, then the nonlinear and linear arc lead to different iterates; see \Cref{fig:projections}. 
Solving the projection problem with the Euclidean metric leads to a suboptimal projection at which the scheme stagnates in the absence of any of the remedies outlined above.
\end{example}

\section{PNKH-B}\label{sec:projectedNewtonMethod}
In this section, we introduce our PNKH-B. In \Cref{subsec:outline_algo}, we present an outline of the algorithm. At each iteration, each backtracking line search requires computing a projection, which is a quadratic program. In \Cref{sec:IPM}, we present the derivation and implementation of an interior point method to solve the quadratic program effectively. In \Cref{sec:activeSet}, we present two variants of our PNKH-B which incorporate the current estimates of the active set.

\subsection{Outline of PNKH-B}\label{subsec:outline_algo}
Our projected Newton-Krylov method with a low-rank approximated Hessian metric (PNKH-B) is a generalized one-metric method that approximately solves the bound-constrained optimization problem~\eqref{eq:optProb}. At each iteration PNKH-B is given by~\eqref{eq:iterHess} and~\eqref{eq:projHess}. 

The global convergence of PNKH-B is guaranteed under standard assumptions; see \Cref{sec:proof_convergence}. We set $\bfH_k$ as a low-rank approximation of the (approximate) Hessian at $\bfx_k$ generated by Lanczos tridiagonalization \cite{lanczos1950}. Specifically, the Krylov subspace is defined by the (approximate) Hessian and gradient at $\bfx_k$. The low-rank approximation is given by $\bfH_k=\bfV_k \bfT_k \bfV_k^\top$, where $\bfV_k \in \mathbb{R}^{ n \times l}$ has orthonormal columns, $\bfT_k \in \mathbb{R}^{l \times l}$ is tridiagonal and $l$ is the rank of the low-rank approximation. 
We slightly abuse notation and denote the pseudoinverse $\bfV_k \bfT_k^{-1} \bfV_k^\top$ by $\bfH_k^{-1}$ in order to be consistent with the conventional notation used in Newton's method. The positive definite matrix $\tilde{\bfH}_k$ used to define the projection norm is obtained by applying a shift in the orthogonal complement of the Krylov subspace to $\bfH_k$. Lanczos tridiagonalization is suitable for large-scale problems because is does not require the explicit (approximate) Hessian $\bfG_k$, but only the function $g_k: \bfy \mapsto \bfG_k\bfy$. Using the low-rank approximation, we effectively compute the pseudoinverse $\bfH_k^{-1}$ and the projection $\Pi_{\| \cdot \|_{\tilde{\bfH}_k}}(\cdot)$, which has to be done once for each line search. The projection problem is solved using an interior point method, which exploits the low-rank approximation effectively and scales only linearly with the number of variables. The interior-point method will be discussed in \Cref{sec:IPM}. The outline of our PNKH-B is summarized in \Cref{alg:projectedNewtonCG_index}.

\subsection{Interior Point Method}\label{sec:IPM}
In this section, we present the derivation and effective implementation of the interior point method tailored to exploit the low-rank structure in~\eqref{eq:projHess}.
\paragraph{Derivation}
We use a standard primal-dual interior point method to solve the projection problem~\eqref{eq:projHess}, which we derive following the outline in \cite[Chapter~16.6]{nocedal2006}. To obtain $\bfx_{k+1}$, we re-formulate~\eqref{eq:projHess} as
\begin{equation*}
    \min_{\bfz,\bfw} \hf \bfz^\top \tilde{\bfH}_k \bfz-\bfz^\top \tilde{\bfH}_k \bfy_{k+1} \quad \text {subject to} \quad \bfK \bfz - \bfb = \bfw \text{ and } \bfw \geq \mathbf{0},
\end{equation*}
where $\tilde{\bfH}_k=\bfV_k \bfT_k \bfV_k^\top + c \bfU_k \bfU_k^\top$ is the low-rank approximation of the Hessian plus a small shift $c>0$ in the orthogonal complement spanned by $\bfU_k \in \mathbb{R}^{n \times (n-l)}$, $\bfw\in\R^{2n}$ is a slack vector and
\begin{equation*}
    \bfK = \begin{bmatrix} \bfI \\ -\bfI \end{bmatrix} \; \text{ and } \; \bfb = \begin{bmatrix} \bfl \\ -\bfu \end{bmatrix}.
\end{equation*}
From this, we obtain the KKT conditions
\begin{subequations}\label{KKT_slack}
\begin{align}
    \tilde{\bfH}_k \bfz-\tilde{\bfH}_k\bfy_{k+1}-{\bfK}^\top \bm{\lambda}=\mathbf{0},&\\
    {\bfK} \bfz - \bfb - \bfw = \mathbf{0},& \\
    w_i \lambda_i = 0,& \;\text{ for } i=1,...,2n, \label{eq:pre_perturb}\\
     \bfw \geq \mathbf{0}\; , \;\bm{\lambda} \geq \mathbf{0}, &
\end{align}
\end{subequations}
where $\bm{\lambda} \in \mathbb{R}^{2n}$ is a vector of Lagrange multipliers. Since the problem (\ref{eq:projHess}) is convex and the interior of the feasible set is non-empty, Slater's condition is satisfied and hence the KKT conditions are necessary and sufficient.
As usual in interior point methods, we consider the perturbed KKT conditions
\begin{equation}\label{perturb_KKT}
    F(\bfz,\bfw,\bm{\lambda};\sigma,\xi) = \begin{bmatrix} \tilde{\bfH}_k \bfz-\tilde{\bfH}_k \bfy_{k+1}-{\bfK}^\top \bm{\lambda} \\ {\bfK} \bfz - \bfb - \bfw \\ \bfW\mathbf{\Lambda} \bfe- \sigma \xi \bfe \end{bmatrix}=\mathbf{0},
\end{equation}
where $\bfW=\text{diag}(\bfw)$, $\mathbf{\Lambda}=\text{diag}(\bm{\lambda})$, $\bfe\in\R^{2n}$ is a vector of ones, $\sigma \in [0,1]$ and $\xi=\bfw^\top \bm{\lambda}/(2n)$ is the duality measure. The solutions of (\ref{perturb_KKT}) define the central path and tend to the solution of~\eqref{KKT_slack} \cite[Section~16.6]{nocedal2006}.

We then apply Newton's method to find the root of the system (\ref{perturb_KKT}). At the $j$th iteration of Newton's method, the step is obtained by solving
\begin{equation}\label{full_KKT_iter}
    \begin{bmatrix} \tilde{\bfH}_k & 0 & -{\bfK}^\top \\ {\bfK} & -\bfI & 0 \\ 0 & \mathbf{\Lambda}_j & \bfW_j \end{bmatrix} \begin{bmatrix} \Delta \bfz_j \\ \Delta \bfw_j \\ \Delta \bm{\lambda}_j 
    \end{bmatrix} = \begin{bmatrix} -\bfr_j \\ -\bfv_j \\ -\bfW_j \mathbf{\Lambda}_j \bfe + \sigma \xi_j \bfe \end{bmatrix},
\end{equation}
which is obtained by differentiating $F$ in (\ref{perturb_KKT}) and setting
\begin{align*}
    \bfr_j & = \tilde{\bfH}_k \bfz_j - \tilde{\bfH}_k \bfy_{k+1} - {\bfK}^\top \bm{\lambda}_j,  \\
    \bfv_j & = {\bfK} \bfz_j - \bfb - \bfw_j. 
\end{align*}
Here, $\bfr_j$ and $ \bfv_j $ are the dual and primal residuals, respectively. After computing $\Delta \bfz_j$, $\Delta \bfw_j$ and $\Delta \bm{\lambda}_j$, the update of the interior point method is
\begin{equation*}
    (\bfz_{j+1}, \bfw_{j+1}, \bm{\lambda}_{j+1}) = (\bfz_j, \bfw_j, \bm{\lambda}_j) + \beta_j (\Delta \bfz_j, \Delta \bfw_j, \Delta \bm{\lambda}_j),
\end{equation*}
where $\beta_j = \text{min} (\beta^{\text{pri}}_{j}, \beta^{\text{dual}}_{j})$ and 
\begin{align*}
    \beta_{j}^{\text{pri}} &= \text{max} \{ \beta \in (0,1] : \bfw_j + \beta \Delta \bfw_j \geq (1-\tau) \bfw_j \}, \\
    \beta_{j}^{\text{dual}} &= \text{max} \{ \beta \in (0,1] : \bm{\lambda}_j + \beta \Delta \bm{\lambda}_j \geq (1-\tau) \bm{\lambda}_j \}.
\end{align*}
The parameter $\tau \in (0,1]$ controls the distance to the boundary of the feasible set. 
While there are other schemes to determine the step size  (see, e.g., \cite{curtis2007}), this simple choice has been effective in our experiments.
\paragraph{Efficient Implementation}\label{sec:IPM_implementation}
The most crucial step of the interior point method is the computation of the solution of the step in~\eqref{full_KKT_iter}. Our implementation exploits the low-rank structure of $\bfH_k$ to directly solve the linear system with $\CO(nl^2)$ floating point operations, where $l$ is the rank of the low-rank approximation; see \Cref{alg:Interior} for an overview.

To compute the update, we first multiply the third equation of (\ref{full_KKT_iter}) by $\mathbf{\Lambda}_j^{-1}$ and add it to the second equation of \eqref{full_KKT_iter} and obtain
\begin{equation}\label{eq:aug}
    \begin{bmatrix} \tilde{\bfH}_k &  -{\bfK}^\top \\  {\bfK} & \mathbf{\Lambda}^{-1}_j\bfW_j \end{bmatrix} \begin{bmatrix} \Delta \bfz_j \\ \Delta \bm{\lambda}_j 
    \end{bmatrix} = \begin{bmatrix} -\bfr_j \\ -\bfv_j - \bfw_j + \sigma \xi_j \mathbf{\Lambda}^{-1}_j \bfe \end{bmatrix}.
\end{equation}
Multiplying the second equation of (\ref{eq:aug}) by ${\bfK}^\top \bfW^{-1}_j \mathbf{\Lambda}_j$ and adding it to the first equation, we obtain
\begin{equation}\label{eq:the_normal}
    (\tilde{\bfH}_k+{\bfK}^\top \bfW^{-1}_j\mathbf{\Lambda}_j {\bfK}) \Delta \bfz_j = -\bfr_j + {\bfK}^\top \bfp_j,
\end{equation}
where $\bfp_j=\bfW^{-1}_j \mathbf{\Lambda}_j (-\bfv_j-\bfw_j+\mathbf{\Lambda}^{-1}_j\sigma \xi_j \bfe)$. Recall that $\tilde{\bfH}_k=\bfV_k \bfT_k \bfV^\top_k+c \bfU_k \bfU_k^\top$ is the low-rank Hessian approximation plus a shift in the orthogonal complement of the Krylov subspace, where $c>0$. This shift is only applied to the projection matrix but not the matrix of the search direction $\bfH_k^{-1}\nabla{f}(\bfx_k)$. This is because the inverse of the shift $c^{-1}$ is large and will dominate the search direction.
The right hand side of \eqref{eq:the_normal} can be computed explicitly in $\CO(n)$ operations. 
Since $\bfU_k\bfU_k^\top = \bfI - \bfV_k\bfV_k^\top$, we can compute $\tilde{\bfH}_k$ as
\begin{equation*}
    \tilde{\bfH}_k=\bfV_k \bfT_k \bfV^\top_k+c \bfU_k \bfU_k^\top=\bfV_k (\bfT_k-c\bfI_l) \bfV^\top_k+c \bfI
\end{equation*}
without explicitly computing $\bfU_k$.
Defining $\bfE_j = c\bfI + {\bfK}^\top \bfW^{-1}_j\mathbf{\Lambda}_j {\bfK}$ and noticing that $\bfE_j$ is diagonal and invertible, we can use the Woodbury matrix identity to invert the left hand side of \eqref{eq:the_normal}. Specifically
\begin{align}\label{eq:Woodbury}
\begin{split}
    (\tilde{\bfH}_k+ \bfE_j)^{-1} 
    &= (\bfV_k (\bfT_k-c\bfI_l) \bfV^\top_k +\bfE_j)^{-1}  \\
    & = \bfE_j^{-1} - \underbrace{\bfE_j^{-1} \bfV_k}_{\in \mathbb{R}^{n\times l}} \underbrace{((\bfT_k-c\bfI_l)^{-1}+\bfV_k^\top \bfE_j^{-1} \bfV_k)^{-1}}_{\in \mathbb{R}^{l\times l}} \underbrace{\bfV_k^\top \bfE_j^{-1}}_{\in \mathbb{R}^{l \times n}},  \\
    \end{split}
\end{align}
where $l$ is the rank of the low-rank approximation. From (\ref{eq:Woodbury}), we see that it requires $\CO(nl^2)$ flops to compute the solution $\Delta \bfz_j$ of (\ref{eq:the_normal}).
After obtaining $\Delta \bfz_j$, we substitute $\Delta \bfz_j$ into (\ref{full_KKT_iter}) and (\ref{eq:aug}) and obtain
\begin{equation*}
    \Delta \bm{\lambda}_j = \bfp_j - \bfW^{-1}_j\mathbf{\Lambda}_j {\bfK} \Delta \bfz_j, 
    \quad \text{ and } \quad 
    \Delta \bfw_j = {\bfK} \Delta \bfz_j + \bfv_j,
\end{equation*}
whose computation require $\CO(n)$ flops.

Overall, exploiting the structure of $\tilde{\bfH}_k$, the interior point method requires $\CO(nl^2)$ flops per iteration, where $n$ is the number of variables.  
\begin{algorithm}[t]
	\begin{algorithmic}[1]
	 \STATE Inputs: low-rank approximation $\bfV_k \bfT_k \bfV^\top_k \approx \nabla^2 f(\bfx_k)$, point to be projected $\bfy_{k+1}$, initial guess $\bfz_0 \in \R^n$, $\bm{\lambda}_0, \bfw_0 \in \R^{2n}$, $\tau \in (0,1]$, $\text{tol}>0$
	 \FOR{$j=0, 1,2,\ldots$}
	    \STATE compute $\bfr_j  = \tilde{\bfH}_k\bfz_j - \tilde{\bfH}_k\bfy_{k+1} - {\bfK}^\top \bm{\lambda}_j$
        \STATE compute $\bfv_j  = {\bfK} \bfz_j - \bfb - \bfw_j$
		\STATE compute $\bfp_j = \bfW^{-1}_j \mathbf{\Lambda}_j (-\bfv_j-\bfw_j+\mathbf{\Lambda}^{-1}_j \sigma \xi_j \bfe)$
		\STATE compute $\bfE_j = c\bfI + {\bfK}^\top \bfW^{-1}_j\mathbf{\Lambda}_j {\bfK}$
		\STATE compute $\Delta \bfz_j = \big( \bfE_j^{-1} - \bfE_j^{-1} \bfV_k ((\bfT_k-c\bfI_l)^{-1}+\bfV_k^\top \bfE_j^{-1} \bfV_k)^{-1} \bfV_k^\top \bfE_j^{-1} )(-\bfr_j + {\bfK}^\top \bfp_j)$
		\STATE   compute  $\Delta \bm{\lambda}_j = \bfp_j - \bfW^{-1}_j\mathbf{\Lambda}_j {\bfK} \Delta \bfz_j$
        \STATE   compute  $\Delta \bfw_j = {\bfK} \Delta \bfz_j + \bfv_j$
        \STATE    compute $\beta_j = \text{min} (\beta^{\text{pri}}_\tau, \beta^{\text{dual}}_\tau)$, where $\beta_{\tau}^{\text{pri}} = \text{max} \{ \beta \in (0,1] : \bfw_j + \beta \Delta \bfw_j \geq (1-\tau) \bfw_j \}$ and $\beta_{\tau}^{\text{dual}} = \text{max} \{ \beta \in (0,1] : \bm{\lambda}_j + \beta \Delta \bm{\lambda}_j \geq (1-\tau) \bm{\lambda}_j \}$
        \STATE update the variables $(\bfz_{j+1}, \bfw_{j+1}, \bm{\lambda}_{j+1}) = (\bfz_j, \bfw_j, \bm{\lambda}_j) + \beta_j (\Delta \bfz_j, \Delta \bfw_j, \Delta \bm{\lambda}_j)$
        \IF{$\|\bfr_j\|_2 < \text{tol} \text{ and } \|\bfv_j\|_2 < \text{tol}$}
         	\STATE break
         \ENDIF
	 \ENDFOR
	 \STATE Output: $\bfz_{j+1}$ approximate projection of $\bfy_{k+1}$ onto $C$ 
	\end{algorithmic}
	\caption{Interior Point Method for the Projection Problem~\eqref{eq:projHess}}
	\label{alg:Interior}
\end{algorithm}

\subsection{Incorporating Estimates of the Active Set}\label{sec:activeSet}
We introduce two variants of PNKH-B that seek to accelerate the convergence by using estimates of the active set.
The intuitive idea is to ignore coordinate dimensions associated with constraints that are currently active during the construction of the low-rank approximation $\bfH_k$.
To this end, we partition $\bfx_k$ into active and inactive components.
To update the inactive coordinates, we exploit curvature information, and to update the active coordinates, we use a scaled projected gradient descent step. 
Our procedure and estimation of the active coordinates are essentially the same as in the two-metric schemes~\cite{Bertsekas1982,Haber2015,schmidt2012}, which crucially rely on this step to ensure convergence.
Being a generalized one-metric scheme, the convergence theory of PNKH-B applies both with and without partitioning. However, in practice it can be advantageous to use estimates of the active set because with variable partitioning, the active/inactive set information is incoporated into the low-rank approximation. Thus the search direction can capture more feasible directions.

At the $k$th iteration, let $\mathcal{A}_k \subset\{1,2,\ldots,n\}$ contain the indices of the components that are estimated to be active and let $m = |\mathcal{A}_k|$. We denote with $\bfR_k \in \R^{m\times n}$ and $\bfP_k\in\R^{(n-m)\times n}$ the projection operators onto the active and inactive set, respectively.
For example, $\bfR_k$ can be constructed by selecting the rows of an identity matrix associated with $\mathcal{A}_k$.
We shall discuss two common choices for constructing $\calA_k$ below.
Given $\bfP_k$ and $\bfR_k$, the intermediate step in~\eqref{eq:iterHess} is
\begin{equation}\label{eq:updatePartition}
    \bfy_{k+1} = \bfx_k - \mu_k \left( \bfP_k^\top \bfF_k^{-1} \bfP_k\nabla f(\bfx_k) + \nu_k^{-1} \bfR_k^\top \bfR_k \nabla f(\bfx_k) \right).
\end{equation}
Here, $\bfF_k$ is a rank-$l$ approximation of the projected (approximate) Hessian $\bfP_k \bfG_k \bfP_k^\top$ and the constant $\nu_k>0$ is used to balance the sizes of both steps.
In practice, this number is often chosen based on the norm of the step for the inactive components. We set $\nu_k=\| \bfR_k \nabla f(\bfx_k)\|_{\infty} / \| \bfF_k^{-1} \bfP_k \nabla f(\bfx_k)\|_\infty$ in our experiments.
One can verify that this leads to the PNKH-B scheme with the Hessian approximation
\begin{equation}\label{eq:Hess_with_partitioning}
    \bfH_k = \begin{pmatrix} \bfP_k^\top & \bfR_k^\top \end{pmatrix}  \begin{pmatrix} \bfF_k & 0 \\ 0 & \nu_k \bfI_m \end{pmatrix} 
    \begin{pmatrix} \bfP_k \\ \bfR_k \end{pmatrix}.
\end{equation}
We use the separability introduced by this construction in the projection, which decouples into using~\eqref{eq:ProjEuclidean} on the active components and using the interior point method on the inactive components.

We obtain two variants of PNKH-B that differ only by the strategy to estimate active and inactive variables.
\paragraph{PNKH-B (boundary index)} Perhaps the most straightforward estimate of the active set is to choose the components that are within an $\epsilon$ margin around the boundary, where $\epsilon>0$ i.e.,
\begin{equation}
    \calA_k^{\rm bound} =  \{ i \, : \, (\bfx_k)_i \leq l_i + \epsilon \; \text{ or } \; (\bfx_k)_i \geq u_i-\epsilon \}.
\end{equation}
This choice has been used successfully in~\cite{Haber2015}.

\paragraph{PNKH-B (augmented index)} As an alternative active set estimation scheme we use the one proposed in~\cite{Bertsekas1982,schmidt2012}. The idea is that in addition to the $\epsilon$ margin, we consider the sign of the partial derivative so that curvature information is used for those constraints predicted to become inactive, i.e.,
\begin{equation}\label{eq:augmented_index}
    \calA_k^{\rm aug} =  \left\{ i \ : \ [(\bfx_k)_i \leq l_i + \epsilon\, \wedge\, \partial_i f(\bfx_k) > 0] \ \text{ or } \ [(\bfx_k)_i \geq u_i-\epsilon\, \wedge\,  \partial_i f(\bfx_k) < 0]\right\}.
\end{equation}
\begin{algorithm}[t]
\begin{algorithmic}[1]
 \STATE Inputs: Initial guess $\bfx_0 \in C$, tolerance xtol and gtol, line search parameter $\alpha \in (0,1)$, and the rank of the low-rank approximation $l$
 \FOR{$k=0, 1,2,\ldots$}
    \STATE select estimate of active set, $\calA_k \in \{\emptyset, \calA_k^{\rm bound}, \calA_k^{\rm aug} \}$, and build projection matrices $\bfP_k$ and $\bfR_k$. 
	\STATE compute $f(\bfx_k)$, $\nabla f(\bfx_k)$ and (approximate) Hessian $\bfG_k \approx \nabla^2 f(\bfx_k)$
	\STATE compute the Lanczos tridiagonalization $\bfF_k =  \bfV_k \bfT_k \bfV_k^\top \approx \bfP_k \bfG_k\bfP_k^\top$ with initial vector $-\bfP_k \nabla f(\bfx_k)$ (use matrix-free implementation)
    \STATE compute the Hessian approximation $\bfH_k$ in~\eqref{eq:Hess_with_partitioning} and the search direction $-\bfH_k^{-1}\nabla f(\bfx_k)$
    \STATE set $\mu=1$
	\FOR{$i=0,1,2,\ldots$}
    \STATE solve the projection $\bfx_t = \Pi_{\| \cdot \|_{\tilde{\bfH}_k}} (\bfx_k - \mu \bfH_k^{-1}\nabla f(\bfx_k))$ (see \Cref{sec:IPM})
     \IF{$f(\bfx_t) < f(\bfx_k)+ \alpha \nabla f(\bfx_k)^\top (\bfx_t-\bfx_k)$ }
     	\STATE set $\bfx_{k+1} = \bfx_t$ and break
     \ELSE
     	\STATE set $\mu = \mu/2$
     \ENDIF
	\ENDFOR
	\IF{$\| \bfx_{k+1}-\bfx_{k}\|_2/\| \bfx_{k}\|_2<\text{xtol}$ or norm of projected gradient $<$ gtol}
	\STATE break
	\ENDIF
     \IF{$\mu=\mu_{\text{old}}$}
    \STATE set $\mu=\text{min}(1.5*\mu,1)$
     \ENDIF
 \ENDFOR
 \STATE Output: approximate solution $\bfx_{k+1} \in C$. 
\end{algorithmic}
\caption{Outline of PNKH-B for solving~\eqref{eq:optProb}}
\label{alg:projectedNewtonCG_index}
\end{algorithm}

\section{Proof of Global Convergence\label{sec:proof_convergence}}
In this section, we introduce and prove the theorem, which guarantees the global convergence of PNKH-B under mild assumptions. We first state the main theorem.
\begin{theorem}[Global Convergence]\label{thm:global_convergence}
Suppose
\begin{enumerate}
    \item $f$ is twice differentiable, and $\nabla f$ is Lipschitz continuous.
    \item $\inf_{\bfx} \{ f(\bfx) | \bfx \in C \}$ is attained, and $C$ is a box.
    \item The norm of projection in our method is induced by $\tilde{\bfH}_k=\bfH_k+c\bfU_k \bfU_k^\top = \bfV_k \bfT_k \bfV_k^\top + c\bfU_k \bfU_k^\top$, the low-rank approximation of the Hessian using Lanczos tridiagonalization plus a positive shift in the orthogonal complement of the Krylov subspace. Moreover, it is symmetric and uniformly positive definite, i.e. $\tilde{\bfH}_k \succeq s\bfI$ for some $s>0$ and for all $k\in \mathbb{N}$.
\end{enumerate}
Then the sequence $\{ {\bfx_k} \}_k$ generated by PNKH-B converges to a stationary point of \eqref{eq:optProb} regardless of the choice of the starting point $\bfx_0\in C$.
\end{theorem}

The assumptions hold for PNKH-B with and without variable partitioning. 
Hence unlike two-metric methods, the convergence of PNKH-B does not hinge upon active/inactive variable partitioning. 
Also, the theorem can also be straightforwardly extended to the preconditioned setting by repeating the same process in the proof.
Moreover, from the theorem, we obtain that our methods globally converge to the optimal solution for convex problems. We now begin to prove the theorem. The proof follows the approach in \cite{lee2014}, which studies proximal Newton-type methods. We first state and prove some lemmas, which will be used to prove the global convergence.
\begin{lemma}[Descent Direction]\label{thm1}
If $f$ is twice differentiable, $\bfH_k=\bfV \bfT \bfV^\top$ is generated by Lanczos tridiagonalization with initial vector $\nabla f(\bfx_k)$, the projection norm is induced by $\tilde{\bfH}_k=\bfH_k+c\bfU_k \bfU_k^\top$, where $\bfU_k$ contains orthonormal basis vectors of the orthogonal complement of the Krylov subspace, and $C$ is a box, then for any $\mu_k>0$, the update step $\bfd_k := \bfx_{k+1} - \bfx_k$ generated by (\ref{eq:iterHess}) and (\ref{eq:projHess}) satisfies
\begin{equation}\label{eq:d_descent_direction}
     \nabla f(\bfx_k)^\top \bfd_k \leq -\frac{1}{\mu_k} \bfd^\top_k \tilde{\bfH}_k \bfd_k.
\end{equation}
Hence the update step $\bfd_k$ is a descent direction.
\end{lemma}
\begin{proof}[Proof of \Cref{thm1}]
By the second projection theorem~\cite[Chapter~9.3]{beck2014}, the iterate $\bfx_{k+1}=\Pi_{\|\cdot\|_{\tilde{\bfH}_k}}\left( \bfy_{k+1}  \right)$ if and only if 
\begin{equation}\label{eq:SecondProjThm}
    (\bfy_{k+1}-\bfx_{k+1})^\top \tilde{\bfH}_k (\bfz-\bfx_{k+1}) \leq 0 \quad \text{for all } \bfz \in C.
\end{equation}
Substituting $\bfy_{k+1} = \bfx_k - \mu_k \bfV_k \bfT_k^{-1} \bfV_k^\top \nabla f(\bfx_k)$ and $\bfz=\bfx_k$ in~\eqref{eq:SecondProjThm}, we obtain
\begin{equation}\label{eq:des_dir_eq1}
    (\bfx_k - \mu_k \bfV_k \bfT_k^{-1} \bfV_k^\top \nabla f(\bfx_k)-\bfx_{k+1})^\top \tilde{\bfH}_k (\bfx_k-\bfx_{k+1})     \leq 0.
\end{equation}
Substituting $\bfd_k=\bfx_{k+1}-\bfx_{k}$ and $\tilde{\bfH}_k \bfV_k \bfT_k \bfV_k^\top + c\bfU_k \bfU_k^\top$ into~\eqref{eq:des_dir_eq1} and by the fact that the columns of $\bfV_k$ are orthogonal to that of $\bfU_k$, we get
\begin{equation}
     \bfd_k^\top \tilde{\bfH}_k \bfd_k  + \mu_k \nabla f(\bfx_k)^\top \left( \bfV_k \bfV_k^\top \right) \bfd_k \leq 0. \label{eq:des_dir_eq2}
\end{equation}
Noting that the Lanczos algorithm is initialized by the normalized gradient vector, the first column of $\bfV_k$ is given by
\begin{equation*}
    \bfv_1 = \frac{\nabla f(\bfx_k)}{\| \nabla f(\bfx_k)\|_2} =  \frac{\nabla f(\bfx_k)}{\gamma}.
\end{equation*}
Consequently, we obtain  
\begin{equation}
    \begin{split}
    \nabla f(\bfx_k)^\top \left( \bfV_k \bfV_k^\top \right) \bfd_k 
    &= 
    \gamma \bfv_1^\top \left( \bfV_k \bfV_k^\top \right) \bfd_k
    \\
    &= \gamma \begin{bmatrix} 1 & 0 \ldots 0 \end{bmatrix} \bfV_k^\top
    \bfd_k
    \\
    &= \gamma \bfv_1^\top \bfd_k
    \\
    &= \nabla f(\bfx_k)^\top \bfd_k.
    \end{split}
    \label{eq:nablafVV_to_nablaf}
\end{equation}
Plugging this into~\eqref{eq:des_dir_eq2}, we finally obtain
\begin{equation}
    \mu_k \nabla f(\bfx_k)^\top \bfd_k \leq  -\bfd^\top_k \tilde{\bfH}_k \bfd_k,
\end{equation}
which is equivalent to~\eqref{eq:d_descent_direction}.
\end{proof}
\begin{lemma}[Armijo Line Search Condition]\label{thm_ls2}
Suppose $C$ is a box, $\nabla f$ is Lipschitz continuous with constant $L>0$, $\tilde{\bfH}_k$'s are symmetric, and $\tilde{\bfH}_k \succeq s\bfI$ for all $k\in \mathbb{N}$ and for some $s>0$, i.e.
\begin{equation*}
    ||\bfz||_{\tilde{\bfH}_k}^2 \geq s ||\bfz||_2^2, \quad \text{for all } \bfz \text{ and for all } k \in \mathbb{N}.
\end{equation*}
For line search parameter $\alpha \in(0,1)$, if step size $\mu_k$ satisfies
\begin{equation*}
    \mu_k \leq \min\left(1, \frac{2s}{L}(1-\alpha) \right),
\end{equation*}
then the following sufficient descent condition is satisfied
\begin{equation}\label{eq:suff_descent_2}
f(\bfx_{k+1}) \leq  f(\bfx_k) + \alpha \nabla f(\bfx_k)^\top (\bfx_{k+1}-\bfx_k).
\end{equation}
\end{lemma}
\begin{proof}[Proof of \Cref{thm_ls2}]
Since $\bfx_{k}, \bfx_{k+1} \in C$, by the Lipschitz continuity of $\nabla f$, we have
\begin{align*}
     f(\bfx_{k+1}) &\leq f(\bfx_k) + \nabla f(\bfx_k)^\top (\bfx_{k+1}-\bfx_k) + \frac{L}{2}  \|\bfx_{k+1}-\bfx_k\|_2^2 \\
     &\leq f(\bfx_k) + \alpha \nabla f(\bfx_k)^\top (\bfx_{k+1}-\bfx_k).
\end{align*}
Here, the first step uses $\mu_k \leq \frac{2s}{L}(1-\alpha)$, $\tilde{\bfH}_k \succeq s\bfI$ and \eqref{eq:d_descent_direction} and in the second step we use $ \nabla f(\bfx_k)^\top (\bfx_{k+1}-\bfx_k) \leq 0$; see \Cref{thm1}. 
\end{proof}
\begin{lemma}\label{thm3}
Suppose the assumptions on $f$, $C$ and $\tilde{\bfH}_k$'s are the same as those in \Cref{thm_ls2}. Also the backtracking Armijo line search scheme is used. Then $\bfx_*$ is a stationary point of~\eqref{eq:optProb} if and only if $\bfx_*$ is a fixed point of our method.
\end{lemma}
\begin{proof}[Proof of \Cref{thm3}]
The iterate $\bfx_*$ is a fixed point of our method if and only if
\begin{equation}\label{eq:FixedPointProj}
    \bfx_* = \Pi_{\| \cdot \|_{\tilde{\bfH}_*}} \left( \bfx_*-\mu_* \bfV_* \bfT_*^{-1} \bfV_*  \nabla f(\bfx_*) \right),
\end{equation}
where $\tilde{\bfH}_*$ is the shifted Hessian approximation, and $\mu_*>0$ is the step size. By the second projection theorem again, it is equivalent to
\begin{align*}
    (\bfx_* - \mu_* \bfV_* \bfT_*^{-1} \bfV_* \nabla f(\bfx_*)-\bfx_*)^\top \tilde{\bfH}_* (\bfz-\bfx_*) &\leq 0 \quad \text{for all } \bfz \in C.
\end{align*}
Using $\tilde{\bfH}_* = \bfV_* \bfT_*^{-1} \bfV_* + c\bfU_*\bfU_*^\top$, and following the same approach as in~\eqref{eq:nablafVV_to_nablaf}, 
this is simplified to $\nabla f(\bfx_*)^\top (\bfz-\bfx_*) \geq 0$ for all $\bfz \in C$, which is true if and only if $\bfx_*$ is a stationary point of the problem.
\end{proof}
Now, we are ready to prove \Cref{thm:global_convergence}, the global convergence of our method.
\begin{proof}[Proof of \Cref{thm:global_convergence}]
The sequence $\{ f({\bfx_k}) \}_k$ is decreasing because the update directions are descent directions (\Cref{thm1}) and the backtracking Armijo line search scheme guarantees sufficient descent at each step (\Cref{thm_ls2}). Since $f$ is closed and its infimum in $C$ is attained, the decreasing sequence $\{ f({\bfx_k}) \}_k$ converges to a limit.

By the sufficient descent condition (\ref{eq:suff_descent_2}), the convergence of $\{ f({\bfx_k}) \}_k$ and $\alpha>0$,
\begin{equation*}
\nabla f(\bfx_k)^\top(\bfx_{k+1}-\bfx_k) 
\end{equation*}
converges to zero. By \Cref{thm1}, one has
\begin{equation*}
    (\bfx_{k+1}-\bfx_k)^\top \tilde{\bfH}_k (\bfx_{k+1}-\bfx_k) \leq -\mu_k \nabla f(\bfx_k)^\top (\bfx_{k+1}-\bfx_k).
\end{equation*}
Hence $(\bfx_{k+1}-\bfx_k)^\top \tilde{\bfH}_k (\bfx_{k+1}-\bfx_k)$ converges to zero. Since $\tilde{\bfH}_k$'s are uniformly positive definite, $\bfx_{k+1}-\bfx_k$ converges to the zero vector.

This implies that the sequence $\{ \bfx_k\}_k$ converges to a fixed point of our method. By \Cref{thm3}, the sequence converges to a stationary point of the problem.
\end{proof}

\section{Experimental Results\label{sec:exp_results}}
We perform three numerical experiments motivated by different applications with PNKH-B. We compare its performance with two state-of-the-art projected Newton-CG (PNCG) methods, which are two-metric schemes; see \Cref{sec:Benchmark}.
In \Cref{sec:EP1}, we consider a PDE parameter estimation problem. 
In \Cref{sec:EP2}, we apply our method to an image classification problem. 
In \Cref{sec:EP3}, we experiment with an image reconstruction problem. 
All these applications require fitting a computational model to data, which is typically noisy.
Therefore, and since the computational models are expensive, we seek to use the optimization scheme to obtain a high-quality reconstruction within only a few iterations.
In all three experiments, using the low-rank approximated Hessian metric during the projection renders PNKH-B competitive with respect to the optimization performance and reconstruction quality to similar state-of-the-art two-metric methods. We set the Hessian shift parameter $c=10^{-3}$.

\subsection{Benchmark Methods}\label{sec:Benchmark} We compare PNKH-B to an implementation of the two-metric scheme described in~\cite{Haber2015} and a variant that includes the augmented indexing scheme from~\cite{Bertsekas1982,schmidt2012}. 
We refer to the scheme obtained using $\calA_k^{\rm bound}$ as PNCG (boundary index) and the scheme obtained using $\calA_k^{\rm aug}$ as PNCG (augmented index); see~\Cref{sec:activeSet}. 
The main difference between these schemes and our proposed method is the projection.
The PNCG schemes use~\eqref{eq:ProjEuclidean} to project all the components and are therefore considered two-metric schemes.
In contrast, our PNKH-B scheme uses a metric that, on the Krylov subspace, is consistent with that implied by the low-rank approximation of the Hessian and is therefore a generalized one-metric scheme. 

\subsection{Experiment 1: Direct Current Resistivity}\label{sec:EP1}
We use PNKH-B to solve the PDE parameter estimation problem motivated by the Direct Current Resistivity (DCR) described in~\cite{Haber2015,ruthotto2017jinv}; see also~\cite{dey1979,mcgillivray1992,wufung2019thesis,wufung2019uncertainty} for background and different instances of this problem.
\paragraph{Model Description\label{subsec:Description_P1}}
The goal of DCR in geophysical imaging is to estimate the conductivity of the subsurface by means of indirect measurement obtained on the earth's surface. Specifically, it first uses electrical sources on the surface to generate direct currents to create electric potential fields in the subsurface. Measurements of these potential fields are then collected on the surface.
The parameter estimation aims at reconstructing a three-dimensional image of the conductivity in the subsurface that is consistent with the measurements; for more details and illustrations of the DCR experiment see, e.g., \cite{dey1979,Haber2015,mcgillivray1992,ruthotto2017jinv}.
\begin{figure}
\includegraphics[width=1\textwidth]{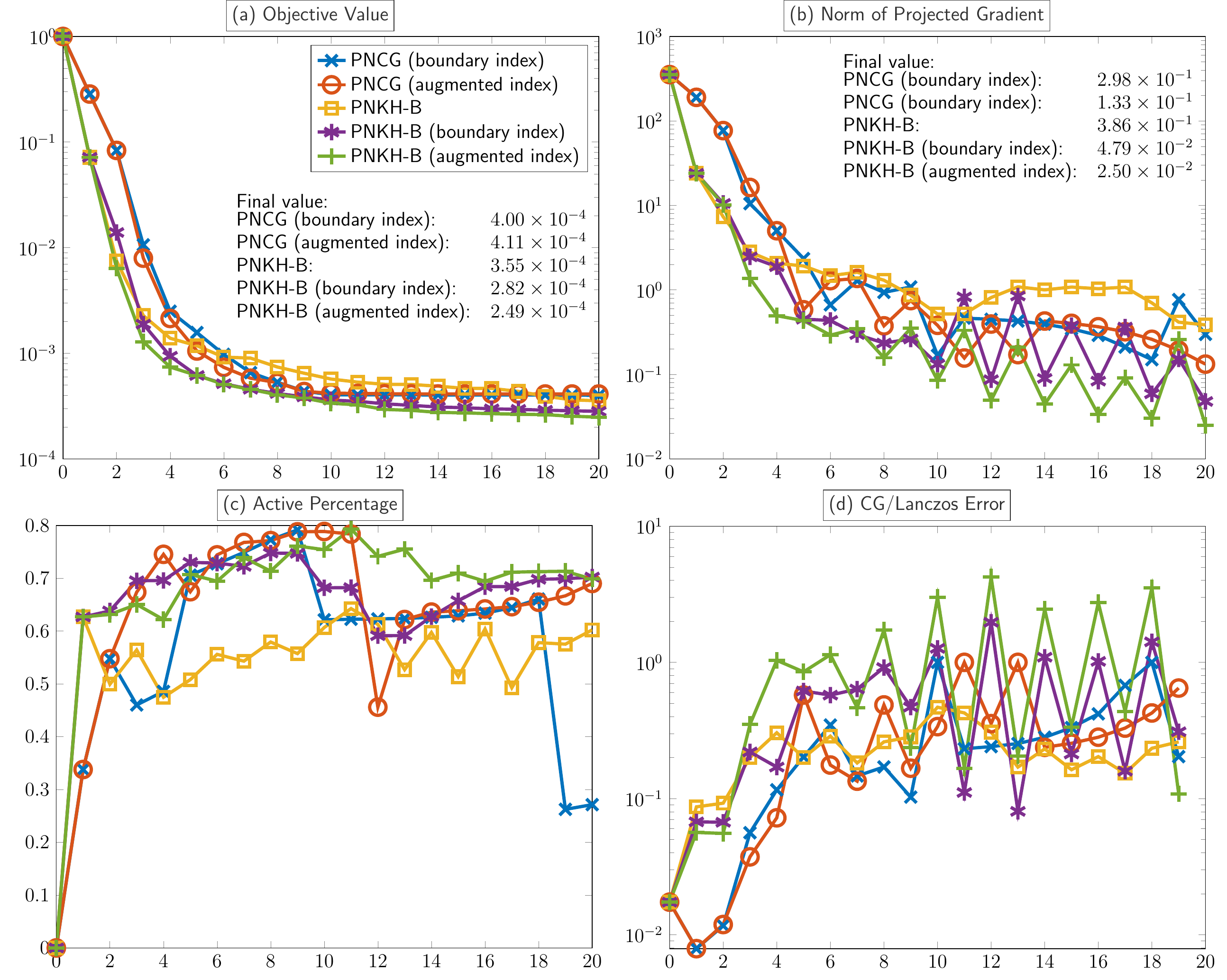}
\caption{Comparison of the convergence of two PNCG methods and three variants of PNKH-B for the direct current resistivity experiment in \Cref{sec:EP1}. (a): Relative reduction of objective function. (b): Norm of the projected gradient. (c): Percentage of variables in $\calA_k^{\rm aug}$ defined in~\eqref{eq:augmented_index}. (d): Relative residual error of CG/Lanczos. \label{fig:DCRfigure1}}
\end{figure}
\begin{figure}[h]
\centering
\includegraphics[width=.9\textwidth]{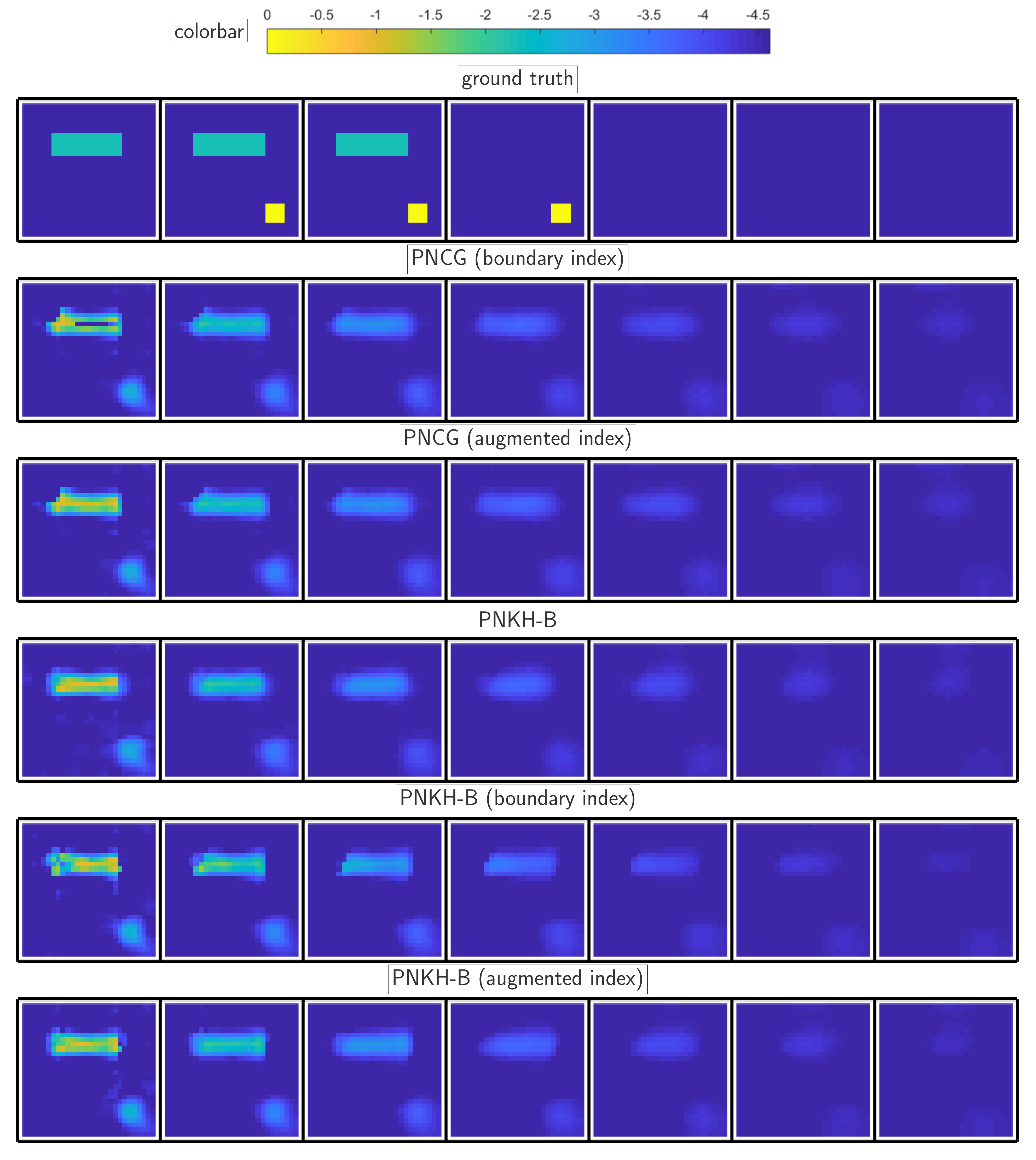}
\caption{Results after the third iteration on DCR generated by the five methods. The upper bound is purposely set to be $\bfm_u=-1$, which is smaller than some pixel values in the ground truth to test the ability of the methods to identify active variables.\label{fig:DCRfigure2}}
\end{figure}
  \begin{figure}[h]
  \centering
\includegraphics[width=1\textwidth]{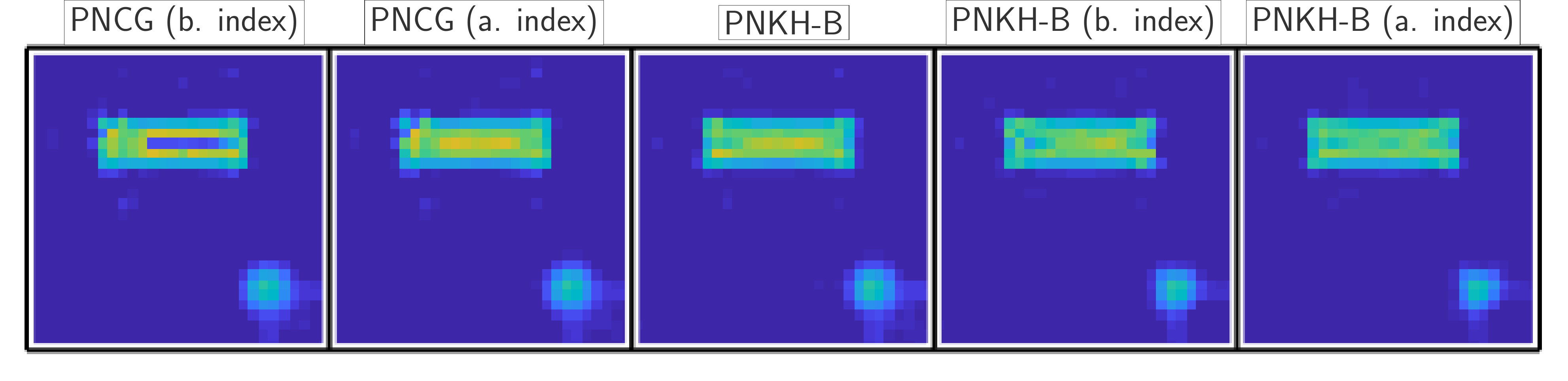}
    \caption{First slice of the final results on DCR generated by the five methods. There are noticeable artifacts in the final results of PNCG.\label{fig:DCRfigure3}}
\end{figure}
\paragraph{Experimental Results\label{subsec:result_P1}}
To set up the test, we follow the same discretize-then-optimize approach described in~\cite{Haber2015} that is also used in~\cite{ruthotto2017jinv,wufung2019multiscale}.
Using a uniform mesh with $N_m$ cells and $N_n$ nodes, we obtain the discrete forward problem
\begin{equation}\label{eq:discrete_forward_DCR}
   \bfD  = \bfP^\top \bfA (\bfm)^{-1} \bfQ + \bm{\epsilon} = \bfP^\top\bfU + \bm{\epsilon},
\end{equation}
where $\bfA(\bfm)\in \mathbb{R}^{N_n \times N_n}$ is a finite-volume discretization of the Poisson operator for the conductivity model $\bfm\in\R^{N_m}$, $\bfP\in \mathbb{R}^{N_n \times N_r}$ is the receiver matrix that maps the fields to data, the columns of $\bfQ\in \mathbb{R}^{N_n\times N_s}$ are discretized sources, the columns of $\bfU\in \mathbb{R}^{N_n\times N_s}$ are the potential fields, and $\bm{\epsilon}\in \mathbb{R}^{N_r\times N_s}$ is Gaussian noise. Here $N_r$ and $N_s$ are the number of receivers and sources, respectively. Note that with suitable discretization and boundary conditions, $\bfA$ is non-singular, which means that $\bfm \mapsto \bfU(\bfm)$ is well-defined and differentiable.

Given the measurement data $\bfD$, sources $\bfQ$, and receivers $\bfP$, we estimate  the corresponding model parameter $\bfm$ by solving the optimization problem
\begin{equation*}
   \min_{\bfm} \frac{1}{2}\|\bfP^\top \bfA (\bfm)^{-1} \bfQ -\bfD\|_F^2 + \frac{\gamma}{2}\| \bfL (\bfm-\bfm_{\text{ref}})\|_2^2 \;\;\;\;
   \text{subject to} \;\;\;\;  \bfm_l \leq \bfm \leq \bfm_u.
\end{equation*}
Here, $\gamma>0$ is a regularization parameter, $\bfm_\text{ref}$ is a given reference model, $\bfL$ is a regularization operator, $\bfm_l$ and $\bfm_u$ are the upper and lower bounds respectively, which are used to enforce the physical constraints for the model parameters. 

As common, we use the Gauss-Newton approximate Hessian $\bfG$ given by
\begin{equation*}
   \bfG = \bfJ(\bfm)^\top \bfJ(\bfm) + \gamma \bfL^\top \bfL,
\end{equation*}
where the Jacobian of the residual of~\eqref{eq:discrete_forward_DCR} is 
\begin{equation}
   \bfJ(\bfm) = - \bfP^\top \bfA(\bfm)^{-1}(\nabla_{\bfm}(\bfA(\bfm)\bfU))^\top.
\end{equation}
Note that the dimensions of $\bfm$ are typically very large. Moreover each evaluation of the objective function or product with the Jacobian $\bfJ$ or its transpose or computing the approximate Hessian-vector multiplication $\bfy \mapsto \bfG\bfy$ require inverting the PDE operator $\bfA$ (i.e., solving the PDE) $\min(N_r,N_s)$ times per source. Hence the computations in each outer (Newton-Krylov method) or inner (line search) iterations when solving the DCR model problem are very expensive, especially when there are a lot of sources.

In this experiment, we solve a 3-dimensional DCR problem on a mesh containing $36 \times 36 \times 12$ cells discretizing the domain $\Omega = (0,1)^3$. The test problem features 25 sources and 1,369 receivers located on the top surface. Following the finite volume discretization presented in~\cite{Haber2015}, we use a cell-centered discretization of the model $\bfm$ and nodal discretizations of the sources, receivers, and fields.
We add $1\%$ noise to the data and enforce smoothness by using a diffusion regularizer with regularization parameter $\alpha = 10^{-3}$. We also use symmetric successive over-relaxation (SSOR) as a preconditioner.

In our setup, we exclude voxels close to the boundary, sources, and receivers from the inversion. As a result, our model $\bfm$ is discretized over $30 \times 30 \times 10$ cells instead; in particular, $\bfm$ has size $n = 9900$.
The bounds $\bfm_l$ and $\bfm_u$ are set as vectors of all -4.6's and -1's, respectively. The upper bound is purposely set as smaller than some pixel values of the ground truth to test the ability of the methods to identify the active variables. 
The main cost of the parameter estimation is the large number of discrete PDE solves to evaluate the objective function, its gradient, and matrix-vector products with $\bfJ$ and $\bfJ^\top$. Therefore, we limit the number of CG/Lanczos to five in all instances. 

The experimental results for the DCR problem are shown in \Cref{fig:DCRfigure1,fig:DCRfigure2,fig:DCRfigure3}. In the experiments, we perform 20 Newton steps. The tolerance in the interior point method is set to be $10^{-10}$. In \Cref{fig:DCRfigure1}(a)-(b), the proposed methods have a significant boost in the initial convergence on the objective function value and the norm of the projected gradient.  
This is particularly evident in the early iterations as can be seen, e.g., by a one-order reduction of the objective function and projected gradient in the second iteration and the visual quality of the parameter estimate at the third iteration; see \Cref{fig:DCRfigure2}. At this iteration, we see that the proposed PNKH-B and PNKH-B (augmented index)'s results are closer to the ground truth and appear smoother. While the results obtained using all methods are similar at the final iteration, we note that PNCG (boundary index) leads to a non-smooth reconstruction; see \Cref{fig:DCRfigure3}.
Since PNCG is a two-metric scheme, the loss of the smoothness might be due to suboptimal scaling of the gradient step in~\eqref{eq:updatePartition} or the inconsistency of the preconditioner caused by the indexing. 

The proposed methods also have slightly smaller objective values after 20 iterations. \Cref{table:runtime} shows that all five methods require a comparable runtime. 
We highlight that the added costs of the interior point method used to compute the projection is only between 0.5\% and 2.5\% and took on average between 0.04 and 0.2 seconds. While the Lanczos tridiagonalization in PNKH-B takes longer on average than the conjugate gradient method in PNCG, the overall runtime of PNKH-B is reduced.
A key reason for the computational savings is the smaller number of backtracking line search iterations, which consequently reduces the number of projections and also PDE solves.

\subsection{Experiment 2: Image Classification}\label{sec:EP2}
We compare the performance of PNKH-B and PNCG for a multinomial logistic regression (MLR) arising in the supervised classification of hand-written digits in the  MNIST dataset~\cite{lecun1998}.

\paragraph{Model Description}\label{subsec:Description_P2}
Let $n_c$ the number of classes, and $\Delta_{n_c}$ be the unit simplex in $\R^{n_c}$.
Denote the training data by $\{(\bfb_j, \bfc_j)\}_{j=1}^N  \subset \mathbb{R}^{n_{\rm f}} \times \Delta_{n_c}$, where $\bfb_j$ is the input feature and $\bfc_j$ is its corresponding class label.
First, given training data $\bfb_j$, we enhance the features by propagating the original features $\bm{b}_j$ into a higher-dimensional space in $\R^{m_{\rm f}}$ obtained by a nonlinear transformation to the original variables. This procedure is also known as extreme learning machines \cite{huang2006}. It improves the ability to classify images that are not from the training set, i.e., it improves the ability to generalize. To perform extreme learning machines, we multiply the vectorized images by a randomly chosen matrix  $\mathbf{K} \in \mathbb{R}^{m_{\rm f} \times n_{\rm f}}$ with  $m_{\rm f}>n_{\rm f}$ whose entries are sampled from a standard normal distribution, i.e.,
\begin{equation*}
    \mathbf{d}_{j} = \tanh(\mathbf{K}\mathbf{b}_j ) \in \mathbb{R}^{m_{\rm f}}.
\end{equation*}
The supervised classification problem aims at training a hypothesis function $h_{\bfX} : \mathbb{R}^{m_{\rm f}} \to \Delta_{n_c}$ that accurately approximates the input-output relationship for new examples, i.e.,
\begin{equation}
    h_{\bfX}(\bfd^{\rm test}_i) \approx \bfc^{\rm test}_i, \quad \text{ for } \quad i=1,\ldots,M.
\end{equation}
Here, $\bfX$ are parameters of the hypothesis function and $\{ \bfd^{\rm test}_i, \bfc^{\rm test}_i\}_{i=1}^M$ is a test dataset, which is not used during training. 
\begin{figure}[h!]
\includegraphics[width=1\textwidth]{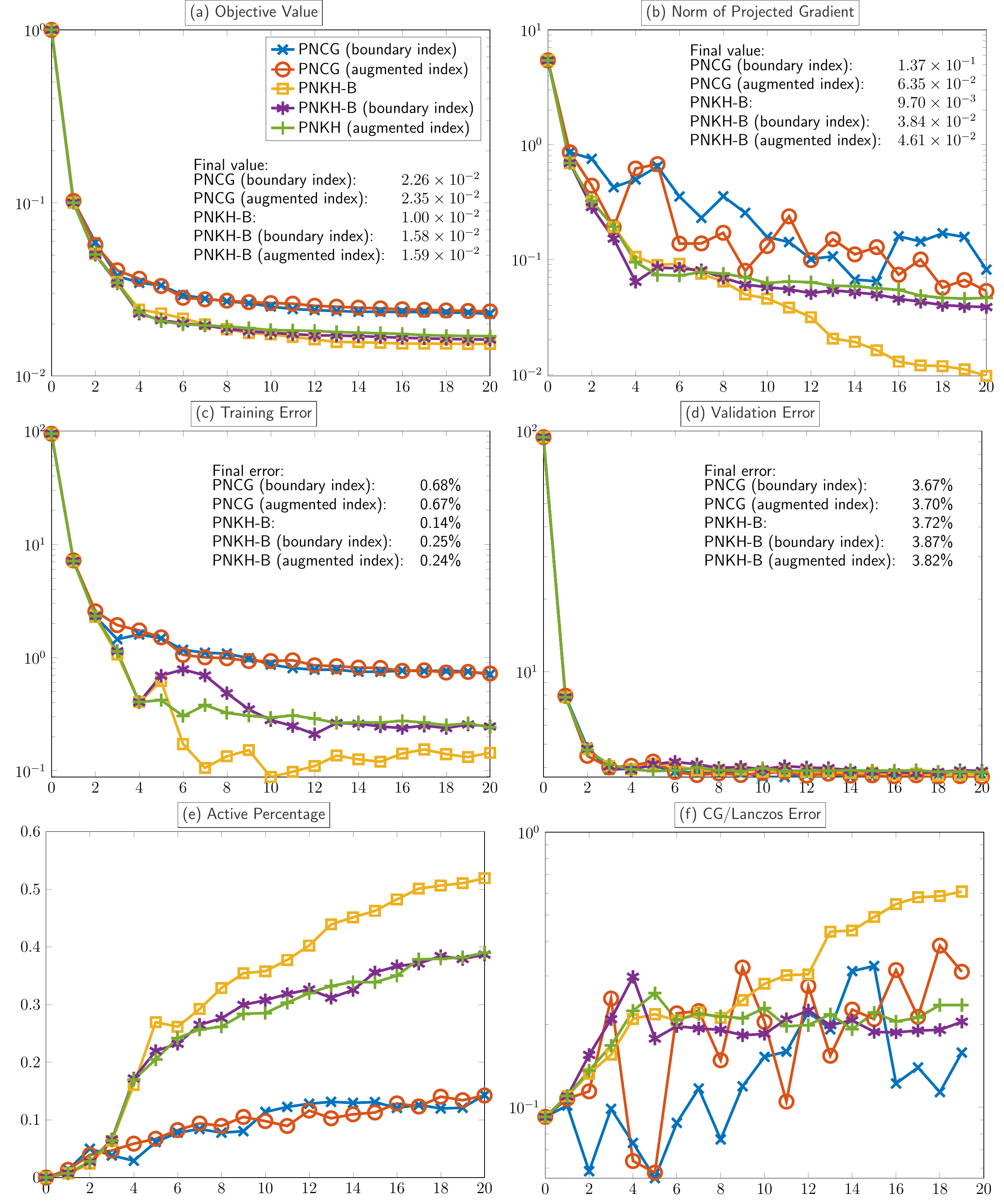}
\caption{Comparison of the convergence of two PNCG methods and three variants of PNKH-B for the image classification problem in \Cref{sec:EP2}. (a): Relative reduction of objective function. (b): Norm of the projected gradient. (c): Training errors. (d): Validation errors. (e): Percentage of variables in $\calA_k^{\rm aug}$ defined in~\eqref{eq:augmented_index}. (f): Relative residual error of CG/Lanczos. \label{fig:MNIST_results}}
\end{figure}

A common strategy for finding the hypothesis function is by solving the MLR problem
\begin{equation}\label{eq:MLR}
     \min_{\bfX_l \leq \bfX \leq \bfX_u} \frac{1}{N} \sum_{j=1}^N -\bfc_j^\top \log\big(h_\bfX (\bfd_j)\big) \quad
   \text{where} \quad   h_\bfX (\bfd_j) = \dfrac{\exp{(\bfX\bfd_j})}{\bfe_{n_c}^\top\exp{(\bfX\bfd_j})}.
\end{equation}
Here, the hypothesis function  is a linear model followed by a softmax transformation, which ensures that $ h_\bfX (\bfd) \in \Delta_{n_c}$ and the objective is to minimize the cross-entropy between the predicted probability distribution and the label. 
In the formulation above, we use $\bfX_l, \bfX_u \in\R^{n_c \times m_{\rm f}}$ to model lower and upper bounds on the entries of $\bfX$, respectively, with the goal to regularize the problem and improve generalization, which means improving the performance on the test data set.
Since the MLR problem is a smooth convex optimization problem, we use $\bfG = \nabla^2f(\bfX)$.

In our experiment, we use the MNIST dataset~\cite{lecun1998}, which consists of $60,000$ $28 \times 28$ grey-scale hand-written images of digits ranging from 0 to 9 that are split into $N=50,000$ training images and $M=10,000$ validation images. Here the transformed feature vectors are in a $m_{\rm f}=4000$-dimensional space.

\paragraph{Experimental Results}\label{subsec:result_P2}
We use a fixed number of 20 inexact Newton steps with 20 CG/Lanczos iterations per step for all five methods and manually tune the bounds on $\bfX$ so that the trained hypothesis function performs well on the validation data. We set the tolerance in the interior point method to be $10^{-12}$. In our case, we choose the entries of $\bfX_l$ and $\bfX_u$ to be -0.05 and 0.05, respectively. The performance of the optimization schemes and the accuracy of the hypothesis function can be seen in  \Cref{fig:MNIST_results}. In particular, in \Cref{fig:MNIST_results}(a)-(c), the three PNKH-B methods boost the initial convergence and outperform the PNCG methods with respect to the objective function value, norm of the projected gradient, and training error by some margin.
The comparison for the validation data is overall comparable; see \Cref{fig:MNIST_results}(d). Despite the more expensive projection step, the PNKH-B variants require a similar runtime in this experiment; see  \Cref{table:runtime}.

\subsection{Experiment 3: Spectral Computed Tomography}\label{sec:EP3}
We consider an image reconstruction problem arising in energy-windowed spectral computed tomography (CT).
The goal is to identify the material composition of an object from measurements taken with x-rays at different energy levels and from different projection angles.  
Our experimental setup follows~\cite{Hu2019,hu2019spectral} which also provide an excellent description and derivation of the problem. 
\paragraph{Model Description}\label{subsec:Description_P3}
As a forward model, we consider the discretized energy-windowed spectral CT model
\begin{equation}\label{eq:vecterize_CT}
    \bfy = (\bfS^\top \otimes \bfI) \text{exp} \{ -(\bfC\otimes \bfA)\bfw \} + \bm{\epsilon},
\end{equation}
where $\bfI \in \mathbb{R}^{(N_d \cdot N_p)\times(N_d \cdot N_p)}$ is the identity matrix, $\bfS \in \mathbb{R}^{N_e \times N_b}$ contains the spectrum energy of each energy window, $\bfC \in \mathbb{R}^{N_e \times N_m}$ contains the attenuation coefficients of each material at each energy level, $\bfA \in \mathbb{R}^{(N_d \cdot N_p)\times N_v}$ contains the lengths of the x-ray beams, $\bfy \in \mathbb{R}^{N_d \cdot N_p\cdot N_b}$ is the observed data containing the x-ray photons of each energy window, $\bfw \in \mathbb{R}^{N_v \cdot N_m}$ represents the weights of the materials of each pixel (and is the unknown variable), and $\bm{\epsilon} \in \mathbb{R}^{N_d \cdot N_p\cdot N_b}$ is the measurement noise. Here, $N_p$ is the number of angles of the x-ray beams, $N_b$ is the number of detectors and each of them detects a specific energy window, $N_m$ is the number of materials, $N_e$ is the number of energy levels of the emitted x-ray beams, $N_d$ and $N_v$ are related to the number of pixels of the image. In particular, for an image of size $n \times n$, $N_d=n$ and $N_v=n^2$.

The goal of the energy-windowed spectral CT model is to estimate the weights of materials $\bfw$ given the other variables except the noise in (\ref{eq:vecterize_CT}). Hence we formulate the following optimization problem
\begin{equation*}
    \min_{0 \leq \bfw \leq \bfw_u} \frac{1}{2}\| \bfy - (\bfS^\top \otimes \bfI) \text{exp} \{ -(\bfC\otimes \bfA)\bfw \}\|_2^2 + \frac{\gamma_1}{2} \| \bfD \bfw_{1:N_v} \|_2^2 + \gamma_2 \sum_{i=N_v+1}^{2N_v} \bfw_{i}.
\end{equation*}
Here, the bound constraints are used to enforce physical bounds, where the weights cannot be negative and cannot exceed the upper bound $\bfw_h$.
The second and third terms are the regularization terms also used in~\cite{hu2019spectral}.
The second term involves the discrete gradient operator $\bfD$ and enforces smoothness of the first material and the last term promotes sparsity of the second material.
As common in nonlinear least-squares problems, we use the Gauss-Newton approximation of the Hessian, i.e., 
\begin{equation*}
    \bfG = \bfJ(\bfw)^\top \bfJ(\bfw) + \gamma_1 \bfD^\top \bfD,
\end{equation*}
where $\bfD$ is a discrete differential operator acting on the first $N_v$ entries.
\begin{figure}[h!]
\includegraphics[width=1\textwidth]{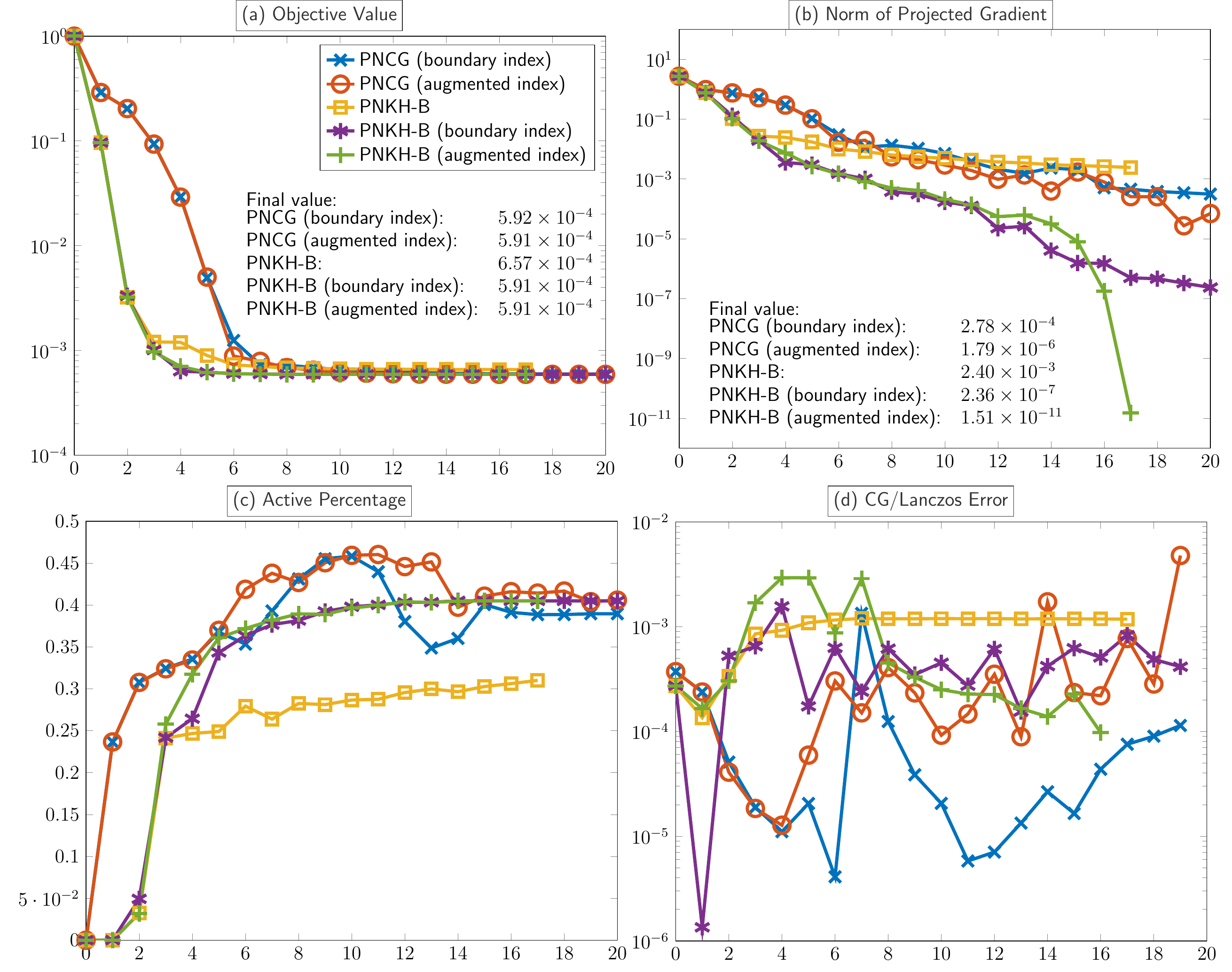}
\caption{Comparison of the convergence of two PNCG methods and three variants of PNKH-B for the energy-windowed spectral CT problem in \Cref{sec:EP3}. (a): Relative reduction of objective function. (b): Norm of the projected gradient. (c): Percentage of variables in $\calA_k^{\rm aug}$ defined in~\eqref{eq:augmented_index}. (d): Relative residual error of CG/Lanczos. \label{fig1:CT_result}}
\end{figure}
\begin{figure}[h!]
\includegraphics[width=1\textwidth]{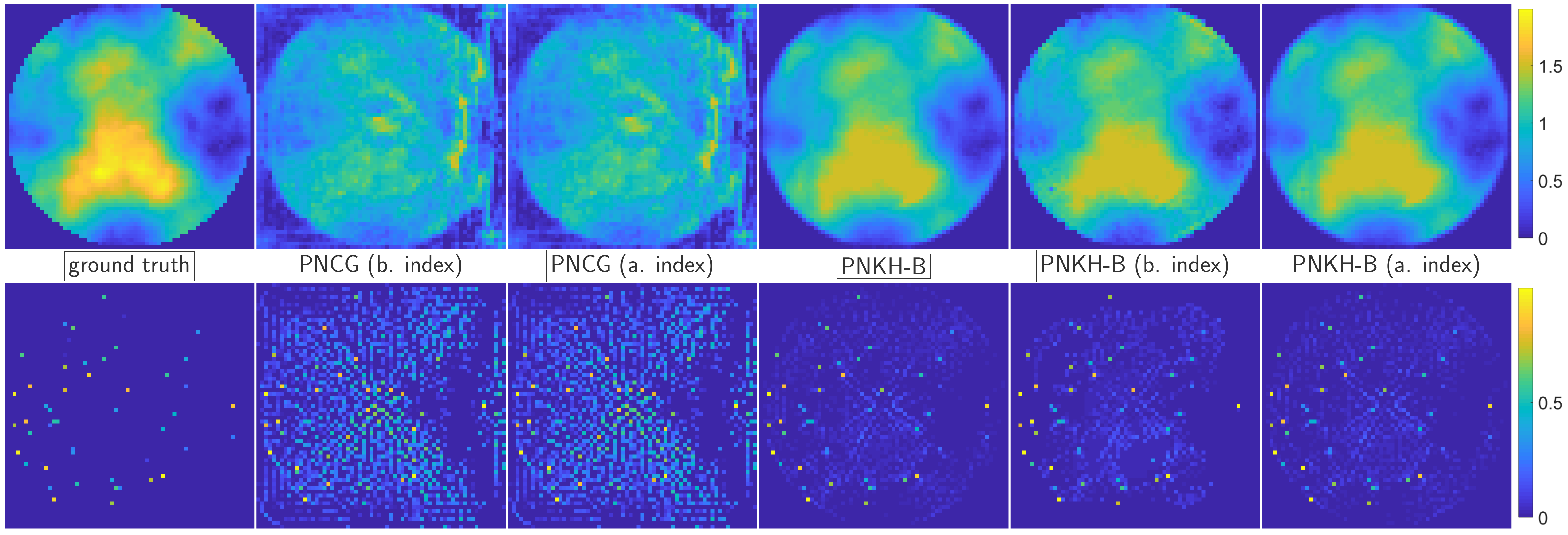}
\caption{Reconstructed images after the second iteration generated by the five methods on CT. The top and bottom images are the estimated composition of the two materials. The upper bound is purposely set to be $\bfw_u=1.5$, which is smaller than some pixel values in the ground truth, to test the ability of the methods to identify active variables. The final image quality is comparable for all schemes.
\label{fig2:CT_result}}
\end{figure}
\paragraph{Experimental Results}\label{subsec:result_P3}
The size of the variables of this problem is $n=8192$. We set the regularization parameters as $\gamma_1=10^9$ and $\gamma_2=10^3$. Since the Kronecker products are implemented effectively, the CT model problem is the least intense among the three testing problems in terms of computational cost. Therefore, we set the number of CG/Lanczos iterations to 100 for all five methods. In this experiment we show that our methods can converge to the optimal solution with very small gradient norm by choosing a very small tolerance of $10^{-16}$ in the interior point method, so that the projection is solved very accurately. Moreover, we purposely choose a tight bound $\bfw_u = [1.5,1.5,...,1.5]^\top$ to test the ability of the methods to compute a solution with many active entries, specifically some entries in the ground truth are outside of this bound. The experimental results of the CT model problem are shown in \Cref{fig1:CT_result,fig2:CT_result}. The proposed methods converge faster initially and all schemes achieve comparable results. In the second iteration of \Cref{fig1:CT_result}(a), the iterate of the three proposed methods achieve 60 times smaller objective function values than the comparing methods.
This also leads to a considerable improvement in the reconstruction quality; see \Cref{fig2:CT_result}. In \Cref{fig1:CT_result}(b), PNKH-B with boundary index and augmented index give competitive performance in terms of the norm of projected gradient. Specifically, they achieve gradient norm with magnitudes $10^{-7}$ and $10^{-11}$, respectively. In this example, the norm of the projected gradient decays very slowly in the last iterations of the PNKH-B scheme. From \Cref{table:runtime}, the runtime of PNKH-B with variable partitioning is roughly double that of PNCG's because of the very low tolerance ($10^{-16}$) in the interior point method, which we use to reduce the norm of projected gradient. We note that this increased accuracy does not necessarily improve the quality of the image reconstruction and in our experiments runtimes similar to that of PNCG can be obtained with a less accurate projection. 
Moreover, the runtime of PNKH-B \emph{without} variable partitioning is roughly three times that of PNCG's because it performs ten backtracking line searches before running into a line search break at iteration 18; the reason is the insufficient decrease in objective function value. 
Finally, we note that further improvements of the optimality conditions can be obtained by adaptively choosing the rank of $\bfH_k$, e.g., by choosing the relative residual tolerance in PCG using a forcing sequence that tightens the tolerance as the solution becomes more accurate.

\begin{table}
\footnotesize
\centering
 \caption{Comparison of runtime in seconds of our prototype implementation, the number of projections, average number of iterations of the interior point method (IPM) and final step length on the three experiments. The sizes of variables of experiment 1, 2 and 3 are 9900, 40010 and 8192, respectively. The total number of iterations is 20. The tests are run on a laptop computer with an Intel Core i5-7200U CPU, 8 GB RAM, and the software platform is MATLAB R2018b.\label{table:runtime}}
{\begin{tabular}{|c|c|c|c|c|c|c|} 
 \hline
 \multicolumn{2}{|l|}{ } & \begin{tabular}{@{}c@{}}PNCG \\ (b. index)\end{tabular} & \begin{tabular}{@{}c@{}}PNCG \\ (a. index)\end{tabular} & \begin{tabular}{@{}c@{}}PNKH-B \\  \end{tabular} & \begin{tabular}{@{}c@{}}PNKH-B \\ (b. index)\end{tabular} & \begin{tabular}{@{}c@{}}PNKH-B \\ (a. index)\end{tabular} \\
 
  \hline
  \multirow{5}{*}{Exp. 1} & Runtime & 192.4.8& 195.9& 195.1 &184.0& 185.3\\
  & IPM time(avg) &  &  & 4.8(0.2) & 0.9(0.04) &0.9(0.05) \\
  & No. of proj. & 31 & 27 & 24  & 20  &20  \\
 &Avg IPM iter. & & &79 &46& 49\\
  &Final step len. & 8.0e-1 & 8.2e-1 & 1.9e-1  &1 & 1 \\
  \hline
    \multirow{5}{*}{Exp. 2} & Runtime & 792.0& 796.5&814.4 &802.3& 791.0\\
  & IPM time(avg) &  &  &13.6(0.47) &9.98(0.33)& 11.23(0.37)\\
  & No. of proj. & 31 & 32 & 29  & 30  &30  \\
 &Avg IPM iter. & & &17 &17& 18\\
  &Final step len. & 1.9e-2 & 9.4e-3 & 1.5e-1  &5.0e-2 & 7.5e-2 \\
  \hline
  \multirow{5}{*}{Exp. 3} & Runtime & 47.3& 46.7&138.5& 83.0&78.5\\
  & IPM time(avg) &  &  &93.1(2.7)&33.6(1.5)&33.1(1.8)\\
  & No. of proj. & 24 & 25 & 34  & 22  &18  \\
 &Avg IPM iter. & & &201 &164& 190\\
  &Final step len. & 1 & 1 & 3.9e-4  &1 & 1 \\
  \hline
\end{tabular}}
\end{table}

\section{Conclusion}\label{sec:conclusion}
We present PNKH-B, a Projected Newton-Krylov method for bound-constrained minimization whose search direction and projection rely on a low-rank approximation of the (approximate) Hessian. Our method can be seen as an extension of Newton-CG methods to bound-constrained problems since we compute the low-rank approximation of the Hessian using a few steps of Lanczos tridiagonalization. The novelty of our method is the use of the metric induced by this approximation in the projection step. We contribute an interior point method that effectively exploits the low-rank approximation to achieve a complexity that is linear with respect to the number of variables.  The consistent use of the metric leads to a simpler algorithm compared to two-metric schemes that require partitioning into active and inactive variables to ensure convergence. We also propose two variants of the framework, which incorporate the current knowledge of the active/inactive variables; this improved the convergence in some cases.  The experimental results on PDE parameter estimation, machine learning and image reconstruction show that the proposed methods lead to faster initial convergence with moderate runtime overhead compared to the existing state-of-the-art projected Newton-CG methods. Our methods are also competitive in the final objective value, norm of the projected gradient and reconstruction quality. We provide our prototype MATLAB code at \url{https://github.com/EmoryMLIP/PNKH-B}.

One direction for future work is to improve the efficiency of the quadratic program solved in the projection stage (e.g., using alternative interior point methods~\cite{ferris2002interior,fine2001efficient,zhang2017modified} or constrained conjugate gradient methods~\cite{more1991solution,yang1991class}).

\section*{Acknowledgments}
The authors would like to thank the two anonymous referees for their thorough review and helpful suggestions and Yunyi Larry Hu and James Nagy for sharing the data and code for the computer tomography experiment.
\bibliographystyle{siamplain}
\bibliography{main}
\end{document}